\documentclass[oneside]{amsart}
\usepackage{amssymb}
\usepackage{amsxtra}
\usepackage{amsmath}
\usepackage{amsthm}
\usepackage[all]{xy}

\usepackage{hyperref}

\usepackage[english]{babel}

\DeclareMathOperator{\HH}{H}

\DeclareMathOperator{\Spec}{Spec}

\DeclareMathOperator{\Br}{Br}

\DeclareMathOperator{\Pic}{Pic}

\DeclareMathOperator{\pr}{pr}

\DeclareMathOperator{\et}{\text{\it \'et}}

\newcommand{\Pro}{\mathbf {P}}
\newcommand{\Aff}{\mathbf {A}}

\newcommand{\ad}{ad}

\newtheorem{lem}{Lemma}
\newtheorem*{thm*}{Theorem}
\newtheorem{thm}{Theorem}
\newtheorem{prop}{Proposition}

\newtheorem*{cor*}{Corollary}
\newtheorem{defn}{Definition}

\newcommand \xra {\xrightarrow }

\newcommand \hra {\hookrightarrow }

\title{Grothendieck---Serre conjecture I: Appendix}
\thanks{A. Stavrova and N.Vavilov are supported by RFBR~09-01-00878; I.Panin is supported by the joint DFG--RFBR project
09-01-91333-NNIO-a}
\author{I. Panin}
\author{A. Stavrova}
\author{N.Vavilov}

\begin{document}

\maketitle

\section{Artin's neighborhood}
In this section we prove Theorem 3.2, Proposition 3.3 and Proposition 3.4 of~\cite{PaSV}.

The following Bertini type theorem is an extension of Artin's
result~\cite[Exp.XI,Thm.2.1]{LNM305}.
\begin{thm}
\label{GeneralSection} Let $k$ be an infinite field, and let
$V\subset \Pro^n_k$ be a locally closed sub-scheme of pure
dimension $r$. Further, let $V^{\prime}\subset V$ be an open
sub-scheme
such that for each point $x \in V^{\prime}$ the scheme $V$ is
$k$-smooth at $x$. Finally, let $p_1,p_2,\dots,p_m\in\Pro^n_k$ be
a family of pair-wise distinct closed points.
For a family $H_1(d),H_2(d),\dots,H_s(d)$, with $s\le r$,
of hypersurfaces of degree $d$ containing all points $p_i$, $1\le i\le m$, set
$ Y=H_1(d)\cap H_2(d)\cap\dots\cap H_s(d)$.

Then there exists an integer $d$ depending on the family
$p_1,p_2,\dots,p_m$ such that if the family
$H_1(d),H_2(d),\dots, H_s(d)$ with $s\leq r$ is sufficiently
general, then $Y$ intersects $V$ transversally at each point of $Y \cap V^{\prime}$.
\par
If, moreover, $V$ is irreducible {\rm(}respectively, geometrically
irreducible{\rm)} and $s<r$ then for the same integer $d$ and for
a sufficiently general family $H_1(d),H_2(d),\dots,H_s(d)$ the
intersection $Y\cap V$ is irreducible {\rm(}respectively,
geometrically irreducible{\rm)}.

\end{thm}

\begin{proof}
Let $\bar k$ be an algebraic closure of $k$. Given a $k$-scheme $X$, we will write $\bar X$ for
$X \otimes_k {\bar k}$.
Let
$p = \coprod p_i$,
let $\bar p$ be the corresponding scheme over
$\bar k$, and let
$q_j$ ($j=1,2, \dots, s$) be all its closed points.

Firstly we prove the Theorem for the case of the field $\bar k$, the $\bar k$-schemes
$\bar V$ and $\bar V^{\prime}$ and the family of closed points
$q_j$ ($j=1,2, \dots, s$). Given that we will be able to choose a sufficiently
general hyperplanes $H_1(d),H_2(d),\dots, H_s(d)$ of the same degree $d$ which are defined over the field $k$
such that the family
$\bar H_1(d),\bar H_2(d),\dots, \bar H_s(d)$
solves our problem for the $\bar k$-schemes
$\bar V$ and $\bar V^{\prime}$ and the family of closed points
$q_j$ ($j=1,2, \dots, s$). Then, clearly, the family
$H_1(d),H_2(d),\dots, H_s(d)$
solves our problem for the $k$-schemes $V$ and $V^{\prime}$ and the family of points
$p_1,p_2,\dots,p_m$.

Now set $q:= \coprod q_i$ and let $H \subset \Pro^n_{\bar k}$ be a hyperplane. Let
$\pi: \tilde \Pro^n_{\bar k} \to \Pro^n_{\bar k}$
be the blow up of
$\Pro^n_{\bar k}$
at $q$ and let $E=\pi^{-1}(q)$ be the exceptional divisor on
$\tilde \Pro^n_{\bar k}$. By
\cite[Ch.II, Sect.7, Prop.7.10]{Ha}
there exists an integer $N >> 0$ such that for each $c > N$
the invertible sheaf
$\mathcal O(cH - E)$
on
$\pi: \tilde \Pro^n_{\bar k} \to \Pro^n_{\bar k}$
is very ample relatively to
$\Pro^n_{\bar k}$.
Take such an integer $c$ and set $d=c+1$.
It is easy to see that the sheaf
$\mathcal O(dH - E)$
is very ample on
$\tilde \Pro^n_{\bar k}$.
It follows that its global sections define a closed
embedding
$\tilde \Pro^n_{\bar k} \hra \Pro^r_{\bar k}$.
One has
\begin{equation*}
\renewcommand{\arraystretch}{1.2}
\begin{array}{rl}
H^0(\Pro^r_{\bar k}, \mathcal O(1))&=H^{0}(\tilde \Pro^n_{\bar k}, \mathcal O(dH - E))\\
&=H^0(\Pro^n_{\bar k}, \pi_*(\mathcal O(dH - E)))\\
&=H^0(\Pro^n_{\bar k}, \pi_*(\mathcal O(-E)) \otimes \mathcal O(dH))\\
&= H^{0}(\Pro^n_{\bar k}, I(d)),
\end{array}
\renewcommand{\arraystretch}{1}
\end{equation*}
where $I$ is the ideal sheaf defining $q$ in $\Pro^n_{\bar k}$.
These identifications of
$H^0(\Pro^r_{\bar k}, \mathcal O(1))$,
$H^{0}(\tilde \Pro^n_{\bar k}, \mathcal O(dH - E))$
and
$H^{0}(\Pro^n_{\bar k}, I(d))$
show that the set of homogeneous polynomials of degree $d$ vanishing at all points of $q$ defines
a locally closed embedding
$f: \Pro^n_{\bar k} - q \hra \Pro^r_{\bar k}$.
Let
$L_i$ be a hyperplane in $\Pro^r_{\bar k}$ corresponding to a degree $d$ hypersurface $H_i$
($i=1,2,\dots, s$)
via the identification of
$H^0(\Pro^r_{\bar k}, \mathcal O(1))$
and
$H^{0}(\Pro^n_{\bar k}, I(d))$.

Set
$\bar V^{\prime\prime}:= \bar V^{\prime}-(\bar V^{\prime} \cap q)$,
$\bar {\bar V}^{\prime\prime}:= f(\bar V^{\prime\prime})$.
Then
$\bar V^{\prime\prime}$
is isomorphic to
$\bar {\bar V}^{\prime\prime}$
and the construction of $f$ implies the following:
if
$L_1 \cap L_2 \cap \dots \cap L_s$ intersects
$\bar {\bar V}^{\prime\prime}$
transversally, then
$Y=H_1 \cap H_2 \dots \cap H_s$
intersects
$\bar V^{\prime\prime}$
transversally.
Let
$\bar {\bar V}$
be the closure of
$\bar {\bar V}^{\prime\prime}$
in
$\Pro^r_{\bar k}$.
Now the item (i) of Theorem
\cite[Exp.XI, Thm. 2.1]{LNM305}
applied to the pair
$\bar {\bar V}^{\prime\prime} \subset \bar {\bar V}$
implies our theorem with
$\bar V^{\prime}$
replaced by
$\bar V^{\prime\prime}$.
It remains to consider points from
$\bar V^{\prime} \cap q$.
At these points $\bar V$ is $\bar k$-smooth.

If we replace the integer $d$ by a larger one, our result is still valid with
$\bar V^{\prime}$ replaced by
$\bar V^{\prime\prime}$.
At the same time, enlarging $d$ we may achieve the following:
among those hypersurfaces of degree $d$ in $\Pro^n_{\bar k}$ that contain $q$,
each one that is sufficiently general, intersects $\bar V$ transversally at every point of $\bar V \cap q$.
This can be easily deduced from the fact that, for a sufficiently
large $d$, the sheaf sequence
$$0 \to I^2(d) \to \mathcal O(d) \to (\mathcal O/I^2)(d)= \mathcal O/I^2 \to 0$$
induces an exact sequence of global sections
$$0 \to H^0(\Pro^n_{\bar k}, I^2(d)) \to H^0(\Pro^n_{\bar k}, \mathcal O(d)) \to H^0(\Pro^n_{\bar k}, \mathcal O/I^2) \to 0.$$

\end{proof}

Next we extend a result of Artin from
\cite{LNM305} concerning existence of nice neighborhoods.
The following notion is due to Artin~\cite[Exp. XI, D\'ef. 3.1]{LNM305}.
\begin{defn}
\label{DefnElemFib} An elementary fibration is a morphism of
schemes $p:X\to S$ that can be included in a commutative diagram
\begin{equation}
\label{SquareDiagram}
    \xymatrix{
     X\ar[drr]_{p}\ar[rr]^{j}&&
\overline X\ar[d]_{\overline p}&&Y\ar[ll]_{i}\ar[lld]_{q} &\\
     && S  &\\    }
\end{equation}
of morphisms satisfying the following conditions:
\begin{itemize}
\item[{\rm(i)}] $j$ is an open immersion dense at each fibre of
$\overline p$, and $X=\overline X-Y$; \item[{\rm(ii)}] $\overline
p$ is smooth projective all of whose fibres are geometrically
irreducible of dimension one; \item[{\rm(iii)}] $q$ is finite
\'{e}tale all of whose fibres are non-empty.
\end{itemize}
\end{defn}

Using Theorem \ref{GeneralSection}, one can prove the following result which is a
slight extension of Artin's result \cite[Exp.XI,Prop.3.3]{LNM305}.

\begin{prop}
\label{ArtinsNeighbor} Let $k$ be an infinite field, $X/k$ be a
smooth geometrically irreducible variety, $x_1,x_2,\dots,x_n\in X$
be closed points. Then there exists a Zariski open neighborhood
$X^0$ of the family $\{x_1,x_2,\dots,x_n\}$ and an elementary
fibration $p:X^0\to S$, where $S$ is an open sub-scheme of the
projective space $\Pro^{\dim X-1}$.
\par
If, moreover, $Z$ is a closed co-dimension one subvariety in $X$,
then one can choose $X^0$ and $p$ in such a way that $p|_{Z\bigcap
X^0}:Z\bigcap X^0\to S$ is finite surjective.
\end{prop}

\begin{proof}
The proof literally follows the proof of the original Artin's result.
Shrinking $X$ may assume that $X \subset \Aff^r$ is affine and still contains the points
$x_1,x_2,\dots,x_n$. Set $x:= \coprod^n_{j=1}x_i$.
Let $X_0$ be the closure of $X$ in $\Pro^r$.
Let $\bar X$ is the normalization of $X_0$ and set $Y=\bar X - X$ with the induced reduced structure.
Let $S \subset \bar X$ be the closed subset of $\bar X$ consisting of all singular points. Then  one has
\begin{itemize}
\item[(i)]
$S \subset Y$,
\item[(ii)]
$\dim \bar X = \dim X = n$,
\item[(iii)]
$\dim Y = n-1$,
\item[(iv)]
$\dim S \leq n-2$.
\end{itemize}
Embed $\bar X$ in a projective space $\Pro^t$. Let $M$ be the restriction of
the invertible sheaf $\mathcal O_{\Pro^t}(t)$ to $\bar X$. Take the sheaf
$M^{\otimes d}$ with the integer $d$ from Theorem \ref{GeneralSection}.
Consider the embedding $\Pro^t \hra \Pro^N$ via the full linear system of hypersurfaces of degree $d$ in $\Pro^t$.
By Theorem \ref{GeneralSection}
there exist hyperplanes
$H_1, H_2, \dots, H_{n-1}$
on
$\Pro^N$
such that
$x \subset H_i$,
$L:= H_1 \cap H_2 \cap \dots \cap H_{n-1}$
has dimension
$N-n+1$
and intersects
$\bar X$ and $Y$ transversally.
The intersection $\bar X \cap L$ is a smooth geometrically irreducible curve  by
Theorem~\ref{GeneralSection}.
The intersection
$Y \cap L$
has dimension zero.
Choose one more hyperplane $H_0$ in $\Pro^N$ such that
$H_0$ intersect $\bar X \cap L$ transversally and
$H_0 \cap Y \cap L = \emptyset$.

Let $h_i$ be a linear form in $N+1$ variables $t_{\nu}$, $\nu=1,\ldots,N+1$, defining the hyperplane $H_i$ in $\Pro^N$.
Consider a rational map
$\Pro^N \xra{q} \Pro^{n-1}$
sending a point $[t_0:t_1: \dots: t_N]$ to the point $[u_0:u_1: \dots:u_{n-1}]$,
where $u_i= h_i(t_0,t_1, \dots, t_N)$. That morphism is a rational projection with the projector center
$C= H_0 \cap \dots \cap H_{n-1}$. Let $\epsilon: \Pro^{\prime} \to \Pro^N$ be the blow up of $\Pro^N$
with the center $C$. Then the diagram
\begin{equation}
    \xymatrix{
    & \Pro^{n-1}   & \\
    \Pro^N \ar[ru]^-{q} && \Pro^{\prime} \ar[ll]_-{\epsilon} \ar[lu]_-{\pi} \\
    }
\end{equation}
commutes as a diagram of rational morphisms and the arrows $\epsilon$ and $\pi$
are regular maps. Let
$\bar X^{\prime} \subset P^{\prime}$
be the proper preimage of $\bar X$, that is
the closure of
$\epsilon^{-1}((\bar X - (\bar X \cap C)))$.
By the hypotheses made $C$ intersects $\bar X$ transversally, and
the morphism $\bar X^{\prime} \to \bar X$
is the blow up at the $k$-smooth finite center $\bar X \cap C$.
The $k$-smoothness means that
the scheme
$\bar X \cap C$
consists of finitely many points, and their residue fields are finite separable field extensions of $k$.

Identify
$X^{\prime}:= X - (\bar X \cap C)$
with an open subscheme of
$\bar X^{\prime}$.
Let
$Y^{\prime}:= \bar X^{\prime} - X^{\prime}$
be the closed subscheme of $\bar X^{\prime}$ with the reduced subscheme structure.
We claim that on the diagram
\begin{equation}
\label{D1}
    \xymatrix{
    X^{\prime} \ar[rd]^-{} \ar[rd]_-{f^{\prime}} \ar[r]^-{j}& \bar X^{\prime} \ar[d]^-{\bar f^{\prime}}  & Y^{\prime} \ar[l]_-{i} \ar[ld]^-{g^{\prime}} \\
    & \Pro^{n-1}  &&  \\
    }
\end{equation}
there exists an open neighborhood $V$ of the point $v=f^{\prime}(x)$ such that the restriction of
the diagram (\ref{D1}) to $V$ satisfies the conditions $(i),(ii),(iii)$ of Definition \ref{DefnElemFib}.
If this claim is true, it completes the proof of the Proposition. We proceed to verify it.

The condition $(i)$ is trivial. To obtain $(ii)$, note that $\bar X \cap L$ is a $k$-smooth geometrically
irreducible curve by the hypotheses. It follows that the morphism
$((f^{\prime})^{-1}(v)) \to \bar X \cap L$
induced by $\epsilon$ is {\it biunivoque}.
Whence
$(f^{\prime})^{-1}(v))_{red} \to \bar X \cap L$
is bijective.
To check that
$\bar f^{\prime}$
is smooth over a neighborhood of $v$, it suffices by
a Lemma of Hironaka
\cite[Exp.II, 2.6]{SGA1}
to verify that $\bar f^{\prime}$ is smooth at a generic point of
$(\bar f^{\prime})^{-1}(v)$.
At that point $\bar X^{\prime}$ is isomorphic to $\bar X$ and the morphism is smooth, since
$L$ intersects $X$ transversally.

It remains to prove that $g^{\prime}$ is \'{e}tale over a neighborhood of $v$.
It is clear that each fibre of $g^{\prime}$ is non-empty since $\dim Y = n-1$.
One has
$\bar X \cap C = \coprod p_i$
scheme-theoretically, where $p_i$ are points on $\bar X$ such that their residue fields
are separable finite extensions of $k$.
One has
$Y^{\prime}= \epsilon^{-1}(Y) \coprod D_1 \coprod \dots, \coprod D_r$,
where
$D_i= \epsilon^{-1}_{\bar X}(p_i) \cong \Pro^{n-1}_{k(p_i)}$.
For each index $i$ the restriction of $g^{\prime}$ to
$\Pro^{n-1}_{k(p_i)}$
coincides with the
morphism
$\Pro^{n-1}_{k(p_i)} \to \Pro^{n-1}_k$
induced by the field extension
$k(p_i)/k$. All these field extensions are finite separable. Whence the morphism
$\Pro^{n-1}_{k(p_i)} \to \Pro^{n-1}_k$
is \'{e}tale.
Clearly, one has
$$q|_{Y} \circ \epsilon|_{\epsilon^{-1}(Y)}=g^{\prime}|_{\epsilon^{-1}(Y)}: \epsilon^{-1}(Y) \to \Pro^{n-1}$$
and
$\epsilon|_{\epsilon^{-1}(Y)}: \epsilon^{-1}(Y) \to Y$
is an isomorphism. The morphism
$q|_{Y}: Y \to \Pro^{n-1}$
is \'{e}tale over a neighborhood of $v$,
since $L$ intersects $Y$ transversally. Whence
$g^{\prime}$
is \'{e}tale over a neighborhood of $v$.

In addition, we may choose $H_0, H_1, \dots, H_{n-1}$ such that
$H_0 \cap H_1 \cap \dots \cap H_{n-1} \cap Z = \emptyset$.
In that case $Z \subset X^{\prime}$ and the morphism
$f^{\prime}|_Z: Z \to \Pro^{n-1}$
is finite surjective. As a consequence the morphism
$f^{\prime}|_{(f^{\prime})^{-1}(V) \cap Z}: (f^{\prime})^{-1}(V) \cap Z \to V$
is finite surjective too.

\end{proof}

\begin{prop}
\label{CartesianDiagram} Let $p: X \to S$ be an elementary
fibration. If $S$ is a regular semi-local scheme, then there
exists a commutative diagram of $S$-schemes
\begin{equation}
\label{RactangelDiagram}
    \xymatrix{
X\ar[rr]^{j}\ar[d]_{\pi}&&\overline X\ar[d]^{\overline \pi}&&
Y\ar[ll]_{i}\ar[d]_{}&\\
\Aff^1\times S\ar[rr]^{\text{\rm in}}&&\Pro^1\times S&&
\ar[ll]_{i}\{\infty\}\times S &\\  }
\end{equation}
\noindent
such that $\bar \pi$ is finite surjective,
$Y=\pi^{-1}(\{\infty\} \times S )$
set-theoretically, and the left hand side square is Cartesian. Here $j$ and
$i$ are the same as in Definition \ref{DefnElemFib}, while
$\pr_S \circ\pi=p$, where $\pr_S$ is the projection
$\Aff^1\times S\to S$.

In particular, $\pi:X\to\Aff^1\times S$ is a finite surjective
morphism of $S$-schemes, where $X$ and $\Aff^1\times S$ are
regarded as $S$-schemes via the morphism $p$ and the projection
$\pr_S$, respectively.
\end{prop}

\begin{proof}
To prove this Proposition it suffices to construct a finite surjective $S$-morphism
$\bar \pi: \overline X \to \Pro^1 \times S$
 such that
$Y=\pi^{-1}(\{\infty\} \times S )$
set-theoretically. To do that, we first note that, under the hypotheses of the Proposition,
$Y \subset \overline X$
is an effective Cartier divisor.
We will construct a desired $\bar \pi$ using two sections $t_0$ and $t_1$ of the sheaf
$\mathcal O(nY)$ for a sufficiently large $n$. Assume that $t_0$ and $t_1$ are
such that the vanishing locus of $t_0$ is $nY$ and the vanishing locus of $t_1$ does not intersect $Y$.
Then the pair $t_0,t_1$ defines a regular map
$\varphi:=[t_0:t_1]: \overline X \to \Pro^1$.
Set
$\bar \pi = (\varphi,\bar p): \overline X \to \Pro^1 \times S$.
Clearly,
$\bar \pi$ is an $S$-morphism of the $S$-schemes. It is a projective morphism
since both $S$-schemes are projective $S$-schemes. It is a quasi-finite surjective morphism.
In fact, for each point $s \in S$ the morphism $\bar \pi$ induces a non-constant morphism
$\overline X_s \to \Pro^1_{s}$
of two $k(s)$-smooth irreducible projective $k(s)$-curves.
Thus $\bar \pi$ is finite surjective as a quasi-finite projective morphism.
It remains to find an appropriate integer $n$ and two sections $t_0$ and $t_1$ with the above properties.

Firstly, for each point $s$ of the scheme $S$ set
$\overline X_s:=(\bar \pi)^{-1}(s)$
scheme-theoretically and note that
$\overline X_s$ is a
$k(s)$-smooth irreducible projective $k(s)$-curve. The morphism $\bar \pi$ is smooth.
In particular, it is flat. Whence the function
$s \mapsto \chi(\overline X_s, \mathcal O_{\overline X_s})$
is constant by Corollary~\cite[Ch.II, Sect.5, Cor.1]{Mu}. The latter means that the genus
$g(\overline X_s)$
is the same for all points $s \in S$. Set
$g = g(\overline X_s)$.

Assume that $n \geq 2g-1$, then
$h^0(\overline X_s, \mathcal O_{\overline X_s}(nY_s))=\chi(\overline X_s, \mathcal O_{\overline X_s}(nY_s))=n-g+1$.
Let
$\mathcal E_n:= \bar p_*(\mathcal O_{\overline X_s}(nY_s)))$.
By
Corollary
\cite[II, Sect.5, Cor.1]{Mu}
and Lemma
\cite[II, Sect.5, Lem.1]{Mu}
the sheaf $\mathcal E_n$ on $S$
is locally free of rank
$n-g+1$,
and for each point $s \in S$ one has
$\mathcal E_n \otimes_{\mathcal O_S} k(s) \cong H^0(\overline X_s, \mathcal O_{\overline X_s}(nY_s))$.

Let $s = \coprod s_i$, where $s_i$ are all closed points of the semi-local scheme $S$. Let
$k(s) = \prod k(s_i)$, where $k(s_i)$ is the residue field of the point $s_i$.
Consider a commutative diagram
\begin{equation}
\label{Diagram3}
    \xymatrix{
H^0(S, \mathcal E_n) \ar[rr]^{id}\ar[d]_{\alpha} && H^0(\overline X, \mathcal O_{\overline X}(nY)) \ar[d]^{\beta}& \\
\mathcal E_n \otimes_{\mathcal O_S} k(s) \ar[rr]^{can} && H^0(\overline X_s, \mathcal O_{\overline X_s}(nY_s)),      & \\  }
\end{equation}
where $\alpha$, $\beta$ and $can$ are the canonical homomorphisms.
As it was mentioned in the previous paragraph, the map $can$ is an isomorphism.
The map $\alpha$ is surjective, since $s$ is a closed subscheme of an affine scheme $S$.
Whence the map $\beta$ is surjective.

For each $s_i \in s$ the curve $X_{s_i}$ is a smooth geometrically irreducible curve of genus $g$.
Whence there exists an integer $n_0$ such that for each $n \geq n_0$ the sheaf
$\mathcal O_{\overline X_s}(nY_s)$
is very ample. By the Bertini type theorem
\cite[Exp.XI,Thm.2.1]{LNM305}
there exists for any $i$ a section
$t_{1,i}$ of $\mathcal O_{\overline X_{s_i}}(nY_{s_i})$ that does not vanish on $Y_{s_i}$.
By the surjectivity of $\beta$ we may choose a section
$t_1$ of $\mathcal O_{\overline X}(nY)$
such that
$\beta(t_1)|_{\overline X_{s_i}}=t_{1,i}$ for any $i$.
The vanishing locus of $t_1$ does not intersect $Y_s$, whence it does not intersect $Y$.
Clearly, $t_1$ is the desired section of
$\mathcal O_{\overline X}(nY)$. It remains to take for $t_0$ a section of $\mathcal O_{\overline X}(nY)$
with the vanishing locus $nY$.

\end{proof}

\section{Equating group schemes}
\label{EquatingGroups}

The following result is Proposition 9.1 of~\cite{PaSV}.
\begin{prop}
\label{PropEquatingGroups} Let $S$ be a regular semi-local
irreducible scheme and let $G_1,G_2$ be two semi-simple
simply connected $S$-group schemes which are twisted forms of each other. Further, let $T\subset S$ be a
closed subscheme of $S$ and $\varphi:G_1|_T\to G_2|_T$ be an
$S$-group scheme isomorphism. Then there exist a finite \'{e}tale
morphism $\tilde S\xra{\pi}S$ together with a morphism
$\delta:T\to\tilde S$ of schemes over $S$, and a $\tilde S$-group scheme
isomorphism $\Phi:\pi^*{G_1}\to\pi^*{G_2}$ such that
$\delta^*(\Phi)=\varphi$.
\end{prop}

Since the proof of the Proposition is rather long we first give an outline.
Clearly, $G_1$ and $G_2$ are of the same type.
By~\cite[Vol.153, Exp.XXIV, Cor.1.8]{SGA3} there exists an $S$-scheme $Isom_S(G_1,G_2)$ representing
the functor that sends an $S$-scheme $W$ to the set of
all $W$-group scheme isomorphisms from $W \times_S G_1$ to $W \times_S G_2$.
The isomorphism $\varphi$
from the hypothesis of Proposition
\ref{PropEquatingGroups}
determines a section
$\delta: T \to Isom_S(G_1,G_2)$
of the structure map
$Isom_S(G_1,G_2) \to S$.
By Lemmas \ref{tildeS}
and \ref{IsomG_1G_2} below
there exists a closed subscheme $\tilde S$ of
$Isom_S(G_1,G_2)$
which is finite \'{e}tale over $S$ and contains $\delta(T)$.
So, we have a commutative diagram of $S$-schemes
\begin{equation}
\label{D}
    \xymatrix{
    T \ar[rrd]_-{i} \ar[rr]^-{\delta}&& \tilde S \ar[rr]^-{} \ar[d]^-{\pi}  && Isom_S(G_1,G_2)  \ar[lld]^-{} \\
    && S  &&.  \\
    }
\end{equation}
such that the horizontal arrows are closed embeddings.
Thus we get an isomorphism
$\Phi: \pi^*(G_1) \to \pi^*(G_2)$
such that
$\delta^{*}(\Phi)=\varphi$.

The precise proof of the Proposition requires some auxiliary results and
will be postponed till the very end of the Section.
Clearly, $G_1$ and $G_2$ are of the same type. Let $G_{0}$ be a split
semi-simple
simply connected algebraic group over the ground field $k$ such that
$G_1$ and $G_2$ are twisted forms of the $S$-group scheme
$S \times_{Spec(k)} G_0$.
Let
$Aut_k(G_0)$
be the automorphism scheme of the algebraic $k$-group $G_0$.
It is known that
$Aut_k(G_0)$
is a semi-direct product of the algebraic $k$-group
$G^{\ad}_0$ and a finite group,
where
$G^{\ad}_0$ is a group adjoint to $G_0$.
Also, $Aut_k(G_0)$ is a smooth affine algebraic $k$-group (for example, by~\cite[Vol.153, Exp.XXIV, Cor.1.8]{SGA3}).
Set for short
$Aut:=Aut_k(G_0)$
and
$Aut_S$ for the $S$-group scheme
$S \times_{Spec(k)} Aut$.

Consider an $S$-scheme $Isom_S(G_{0,S},G_2)$ constructed in
\cite[Vol.153, Exp.XXIV, Cor.1.8]{SGA3}
and representing a functor that sends an $S$-scheme $W$ to the set of
all $W$-group scheme isomorphisms
$\varphi_2:W \times_S G_{0,S}\to W \times_S G_2$.
Similarly, consider
an $S$-scheme $Aut_S(G_2)$ constructed in
\cite[Vol.153, Exp.XXIV, Cor.1.8]{SGA3}
and representing a functor that sends an $S$-scheme $W$ to the set of
all $W$-group scheme automorphisms
$\alpha: W \times_S G_2 \to W \times_S G_2$.

The functor transformation
$(\varphi_2, \alpha_2) \mapsto \varphi_2 \circ \alpha^{-1}_2$
defines an $S$-scheme morphism
$$Isom_S(G_{0,S},G_2) \times_S Aut_S \to Isom_S(G_{0,S},G_2)$$
which makes the $S$-scheme
$Isom_S(G_{0,S},G_2)$
a principal right $Aut_S$-bundle.
The functor transformation
$(\beta_2, \varphi_2) \mapsto \beta_2 \circ \varphi_2$
defines an $S$-scheme morphism
$$Aut_S(G_2) \times_S Isom_S(G_{0,S},G_2) \to Isom_S(G_{0,S},G_2)$$
which makes the $S$-scheme
$Isom_S(G_{0,S},G_2)$
a principal left $G_2$-bundle.

Similarly,
the functor transformation
$(\alpha_1, \varphi_1) \mapsto \alpha_1 \circ \varphi_1$
makes the $S$-scheme
$Isom_S(G_1,G_{0,S})$
a principal left $Aut_S$-bundle and
the functor transformation
$( \varphi_1, \beta_1) \mapsto \varphi_1 \circ \beta_1$
makes the $S$-scheme
$Isom_S(G_1,G_{0,S})$
a principal right $G_1$-bundle.

Let $_{2}P_r$ be a left principal $G_2$-bundle and at the same time a right principal $Aut_S$-bundle
such that the two actions commute.
Let $_{l}P_1$ be a left principal $Aut_S$-bundle and at the same time a right principal $G_1$-bundle
such that the two actions commute.
Let $Y$ be a $k$-variety equipped with a left and a right $Aut_k$-actions which commute. Then the $k$-scheme
$$({_{2}P_r}) \times_S (Y_S) \times_S ({_{l}P_1 })$$
is equipped with a left $Aut_k \times Aut_k$-action given by
$(\alpha_2, \alpha_1)(p_2,y,p_1)=(p_2 \alpha^{-1}_2, \alpha_2 y \alpha^{-1}_1, \alpha_1 p_1)$.
The orbit space does exist (it can be constructed by descent). Denote it by $_{2}Y_1$.
We now show that it is an $S$-scheme. Indeed, the structure morphism
$Y \to Spec(k)$
defines a morphism
$$({_{2}P_r}) \times_S (Y_S) \times_S ({_{l}P_1 }) \to ({_{2}P_r}) \times_S ({_{l}P_1 })$$
respecting the $Aut_k \times Aut_k$-actions on both sides. Thus it defines a morphism of the orbit spaces
$$_{2}Y_1 \to (_{2}Spec(k)_1)=S.$$
The latter equality holds since
$({_{2}P_r}) \times_S ({_{l}P_1 })$
is a principal left $Aut \times Aut$-bundle with respect the left action given by
$(\alpha_2, \alpha_1)(p_2,p_1)=(p_2 \alpha^{-1}_2, \alpha_1 p_1)$.

The construction
$Y \mapsto \ _{2}Y_1$
has several nice properties. Namely,
\begin{itemize}
\item[(i)]
it is natural with respect to $k$-morphisms of $k$-varieties
$Y \to Y^{\prime}$
commuting with the given two-sided $Aut \times Aut $-actions on $Y$ and $Y^{\prime}$,
\item[(ii)]
it takes closed embeddings to closed embeddings,
\item[(iii)]
it takes open embeddings to open embeddings,
\item[(iv)]
it takes $k$-products to $S$-products,
\item[(v)]
locally in the \'{e}tale topology on $S$, the $S$-schemes
$Y_S$ and $_{2}Y_1$ are isomorphic.
\end{itemize}
Set
${_{2}P_r}= Isom_S(G_{0,S},G_2)$
and
${_{l}P_1}= Isom_S(G_1,G_{0,S})$.
The functor transformation
$(\varphi_2, \alpha, \varphi_1) \mapsto \varphi_2 \circ \alpha \circ \varphi_1$
gives a morphism of representable $S$-functors
$$Isom_S(G_{0,S},G_2) \times_S (Aut_S) \times_S Isom_S(G_1,G_{0,S}) \xra{\displaystyle \Phi} Isom_S(G_1,G_2).$$
The equality
$$\varphi_2 \circ \alpha \circ \varphi_1=
(\varphi_2 \circ \alpha^{-1}_2)\circ (\alpha_2 \circ \alpha \circ \alpha^{-1}_1) \circ (\alpha_1 \circ \varphi_1)$$
shows that the morphism $\Phi$ induces a morphism
$\bar \Phi: \ _{2}(Aut_S)_1 \to Isom_S(G_1,G_2)$.

\begin{lem}
\label{IsomG_1G_2}
The $S$-morphism
$$\bar \Phi: \  _{2}(Aut_S)_1 \to Isom_S(G_1,G_2)$$
is an isomorphism.
\end{lem}

\begin{proof}
It suffices to prove that $\bar \Phi$ is an isomorphism locally in the \'{e}tale topology on $S$.
The latter follows from the property (v).
\end{proof}

Now let $G_0$ and $Aut$ be as above. There is a closed embedding of algebraic groups
$\rho: Aut \hra GL_{V,k}$
for an $n$-dimensional $k$-vector space $V$.
Replacing
$\rho$ with $\rho \oplus det^{-1} \circ \rho$ we get a closed embedding of algebraic $k$-groups
$\rho_1: Aut \hra SL_{W,k}$,
where $W=V \oplus k$.
Let $End:=End_k(W)$.
Clearly, the composition
$in: Aut \xra{\rho_1} SL_{W,k} \hra End$
is a closed embedding.
We will identify
$Aut$ with its image in $End$.
Let $\overline {Aut}$ be the closure of $Aut$ in the projective space $\Pro(k \oplus End )$.
Set
$Aut_{\infty}:=\overline {Aut}- Aut$
regarded as a reduced scheme. So, we get a commutative diagram of $k$-varieties
\begin{equation}
\label{RactangelDiagram}
    \xymatrix{
Aut \ar[rr]^{j}\ar[d]^{in}&& \overline {Aut} \ar[d]^{\overline {in}}&& Aut_{\infty} \ar[ll]_{i}\ar[d]^{in_{\infty}}&\\
End \ar[rr]^{\text{\rm J}}&& \Pro(k \oplus End ) && \ar[ll]_{I} \Pro(End)  &\\  }
\end{equation}
where the left square is Cartesian. All varieties are equipped with the left
$Aut \times Aut$-action induced by
$Aut \times Aut$-action on the affine space $k \oplus End$
given by
$(g_1,g_2)(\alpha,c)= (c, g_1 \alpha g^{-1}_2)$.
All the arrows in this
diagram respect this action. Applying to this diagram the above construction
$Y \mapsto \ _{2}Y_1$,
we obtain a commutative diagram of $S$-schemes
\begin{equation}
\label{RactangelDiagram-S}
    \xymatrix{
_{2}Aut_1 \ar[rr]^{j}\ar[d]^{in}&& _{2}(\overline {Aut})_1 \ar[d]^{\overline {in}}&& _{2}(Aut_{\infty})_1 \ar[ll]_{i}\ar[d]^{in_{\infty}}&\\
_{2}End_1  \ar[rr]^{\text{\rm J}}&& \Pro(\mathcal O_S \oplus \ _{2}End_1 ) && \ar[ll]_{I} \Pro(_{2}End_1)  &\\  }
\end{equation}
where the square on the left is Cartesian.

From now on we assume that $S$ is a semi-local irreducible scheme. Then the vector bundle
$_{2}End_1$ is trivial.
Since it is trivial, we may choose homogeneous coordinates $Y_i$'s on
$\Pro(\mathcal O_S \oplus \ _{2}End_1)$
such that the closed subschemes
$\{Y_0=0 \}$
and
$\Pro(_{2}End_1)$
of the scheme
$\Pro(\mathcal O_S \oplus \ _{2}End_1)$
coincide and the $S$-scheme
$\Pro(\mathcal O_S \oplus \ _{2}End_1)$
itself is isomorphic to the projective space
$\Pro^{n^2}_S$. Thus the diagram
(\ref{RactangelDiagram-S})
of $S$-schemes and of $S$-scheme morphisms can be rewritten as follows
\begin{equation}
\label{Compactification}
    \xymatrix{
_{2}Aut_1 \ar[rr]^{j}\ar[d]^{in}&& _{2}(\overline {Aut})_1 \ar[d]^{\overline {in}}&& _{2}(Aut_{\infty})_1 \ar[ll]_{i}\ar[d]^{in_{\infty}}&\\
\{Y_0 \neq 0 \}   \ar[rr]^{\text{\rm J}}&& \Pro^{n^2}_S && \ar[ll]_{I} \{Y_0 = 0 \}  &\\  }
\end{equation}
where the square on the left is Cartesian.
Since
$_{2}(Aut_{\infty})_1= \ _{2}(\overline {Aut})_1 - \ _{2}Aut_1$,
the set-theoretic intersection
$_{2}(\overline {Aut})_1 \cap \{Y_0 = 0 \}$
in $\Pro^{n^2}_S$ coincides with
$_{2}(Aut_{\infty})_1$.

The following Lemma is the lemma~\cite[Lemma 7.2]{OjPa}.
\begin{lem}
\label{Lemma7_2}
Let $S=Spec(R)$ be a regular semi-local scheme and $T$ a closed
subscheme of $S$. Let $\bar X$ be a closed subscheme of
$\Pro^{N}_S= Proj(S[Y_0,\dots,Y_{N}])$ and
$X=\bar X\cap\Aff^{N}_S$,
where
$\Aff^{N}_S$
is the affine space defined by
$Y_0\neq0$. Let
$X_{\infty}=\bar X\setminus X$ be the intersection of $\bar X$ with the
hyperplane at infinity
$Y_0=0$. Assume further that
\begin{itemize}
\item[(1)]$X$ is smooth and equidimensional over $S$,
of relative dimension $r$.
\item[(2)]
For every closed
point $s\in S$ the closed fibres of $X_\infty$ and $X$
satisfy
$$\dim (X_\infty(s))< \dim (X(s))=r\;.$$
\item[(3)]
Over $T$ there exists a section $\delta:T\to X$
of the canonical projection $X\to S$.
\end{itemize}
Then there exists a closed subscheme $\tilde S$ of $X$ which is finite
\'etale
  over
$S$ and contains $\delta(T)$.
\end{lem}
The diagram
(\ref{Compactification})
shows that the $S$-schemes
$X= \ _{2}Aut_1$,
$\bar X = \ _{2}(\overline {Aut})_1$
and
$X_{\infty}= \ _{2}(Aut_{\infty})_1$
satisfy all the hypotheses of Lemma
\ref{Lemma7_2}
except possibly the conditions (2) and (3). To check (2), observe that the diagram of $S$-schemes
\begin{equation}
\label{RactangelDiagram}
    \xymatrix{
_{2}Aut_1  \ar[rr]^{\text{\rm j}}&& _{2}(\overline {Aut})_1  && \ar[ll]_{i} \ _{2}(Aut_{\infty})_1  &\\  }
\end{equation}
locally in the \'{e}tale topology on $S$ is isomorphic to the diagram of $S$-schemes
\begin{equation}
\label{RactangelDiagram}
    \xymatrix{
Aut \times S \ar[rr]^{\text{\rm j}}&& (\overline {Aut}) \times S  && \ar[ll]_{i} (Aut_{\infty}) \times S.  &\\  }
\end{equation}
This follows from the property (v) of the construction $Z \to {}_{2}Z_1$.
Since $Aut$ is equidimensional and $\overline {Aut}$ is the closure of $Aut$ in $\Pro(End \oplus k)$, one has
$$\dim (Aut_{\infty}) < \dim (\overline {Aut})= \dim Aut.$$
Thus the assumption (2) of Lemma~\ref{Lemma7_2}
is fulfilled. Whence we have proved the following

\begin{lem}
\label{tildeS}
Assume $S$ is a regular semi-local irreducible scheme and assume we are given with
a closed subscheme $T \subset S$ equipped with a section
$\delta: T \to \ _{2}Aut_1 $
of the structure map
$_{2}Aut_1 \to S$. Then there exists a closed subscheme
$\tilde S$ of $_{2}Aut_1$
which is finite and \'{e}tale over $S$ and contains $\delta(T)$.
\end{lem}

\begin{proof}[Proof of Proposition \ref{PropEquatingGroups}]
By Lemma
\ref{IsomG_1G_2}
the $S$-schemes
$Isom_S(G_1,G_2)$
and
$_{2}Aut_1$
are naturally isomorphic as $S$-schemes.
The isomorphism
$\varphi$
from the hypotheses of the Proposition
\ref{PropEquatingGroups}
determines a section
$\delta: T \to Isom_S(G_1,G_2)= \ _{2}Aut_1$
of the structure map
$Isom_S(G_1,G_2)= \ _{2}Aut_1 \to S$.
By Lemma
\ref{tildeS}
there exists a closed subscheme $\tilde S$ of
$_{2}Aut_1= Isom_S(G_1,G_2)$
which is finite \'{e}tale over $S$ and contains $\delta(T)$.
So, we have morphisms (even closed inclusions) of $S$-schemes
\begin{equation}
\label{D}
    \xymatrix{
    T \ar[rrd]_-{i} \ar[rr]^-{\delta}&& \tilde S \ar[rr]^-{} \ar[d]^-{\pi}  && Isom_S(G_1,G_2)  \ar[lld]^-{} \\
    && S  &&.  \\
    }
\end{equation}
Thus we get an isomorphism
$\Phi: \pi^*(G_1) \to \pi^*(G_2)$
such that
$\delta^{*}(\Phi)=\varphi$.

\end{proof}

\section{Horrocks type theorem}
In this section we give a proof of Theorem 8.6 of~\cite{PaSV}.

In the Theorem, let $k$ be an infinite field, $\mathcal O$ be the semi-local ring of finitely many points on
a $k$-smooth affine scheme $X$;
$x= \{x_1,x_2,\dots,x_n\}  \subset Spec(\mathcal O)$
be the set of all closed points,
$l=\prod^r_{i=1}k(x_i)$.
Let $G$ be a simple
simply-connected $\mathcal O$-group scheme,
$G_l = G \otimes_{\mathcal O} l$, \ $\Pro^1:=\Pro^1_{\mathcal O}$, $\Aff^1:= \Aff^1_{\mathcal O}$.


\begin{thm}
\label{ConstantOnFibreIsConstant} Let $G$ be an simple
simply connected $\mathcal O$-group scheme. Let $E$ be a principal
$G$-bundle over $\Pro^1$ whose restriction to each closed fibre is trivial, that is
$E_{\Pro^1_l}$ is trivial over $\Pro^1_l$. Then $E$ is of the form:
$E=\pr^*(E_0)$, where $E_0$ is a principal $G$-bundle over
$\Spec(\mathcal O)$ and $\pr:\Pro^1\to\Spec(\mathcal O)$ is the
canonical projection.
\end{thm}
The proof of this Theorem is rather standard,
and for the most part follows \cite{R1}. However, our group scheme
$G$ does not come from the ground field $k$. Therefore, we have to
somewhat modify Raghunathan's arguments. We will use the following lemma.

\begin{lem}
\label{Horrocks}
Let $W$ be a semi-local irreducible Noetherian scheme over an arbitrary field $k$.
Let $H$ and $H'$ be two reductive group schemes over $W$, such that $H$ is a closed $W$-subgroup scheme of $H'$,
and denote by $j: H \hra H^{\prime}$ the corresponding embedding. Denote by $\Pro^1_W$ the projective line over
$W$.

Let $F \in H^1(\Pro^1_W, H)$
be a principal $H$-bundle, and let
$M :=j_*(F) \in H^1(\Pro^1_W, H^{\prime})$
be the corresponding principal $H^{\prime}$-bundle.
If $M$ is a trivial $H^{\prime}$-bundle, then there exists a principal
$H$-bundle $F_0$ over $W$ such that
$pr^*(F_0) \cong F$,
where
$pr: \Pro^1_W \to W$
is the canonical projection.
\end{lem}
\begin{proof}
Set $X= H^{\prime}/j(H)$. Locally in the \'{e}tale topology on $W$
this scheme is isomorphic to the $W$-scheme $W \times_{\Spec(k)} H^{\prime}_{0,k}/H_{0,k}$, where
$H_{0,k}$ and $H^{\prime}_{0,k}$ are
the split reductive $k$-group schemes of the same types as $H$ and $H^{\prime}$ respectively.
By results of Haboush
\cite{Hab}
and Nagata
\cite{Na}
(see Nisnevich \cite[Corollary]{Ni})
the $k$-scheme
$H^{\prime}_{0,k}/H_{0,k}$
is an affine $k$-scheme. Thus $X$ is an affine $W$-scheme.
Consider the long exact sequence of pointed sets
$$1 \to H(\Pro^1_W) \xra{j_*} H^{\prime}(\Pro^1_W) \to X(\Pro^1_W)
\xra{\partial} H^1_{\text{\'{e}t}}(\Pro^1_W, H) \xra{j_*} H^1_{\text{\'{e}t}}(\Pro^1_W, H^{\prime}).$$
Since $j_*(F)$ is trivial, there is $\varphi \in X(\Pro^1_W)$ such that
$\partial (\varphi)= F$.

The $W$-morphism $\varphi: \Pro^1_W \to X$ is a $W$-morphism of a $W$-projective scheme to a $W$-affine scheme.
Thus $\varphi$ is "constant", that is, there exists a section $s: W \to X$ such that
$\varphi= s \circ pr$.
Consider another long exact sequence of pointed sets, this time the one corresponding to the scheme $W$,
and the morphism of the first sequence to the second one induced by the projection $pr$. We get a big commutative diagram.
In particular, we get the following commutative square
\begin{equation}
\label{Representation1}
    \xymatrix{
X(W) \ar[d]_{pr^*_W} \ar[rr]^{\partial} && H^1_{\text{\'{e}t}}(W, H) \ar[d]^{pr^*_W} \\
X(\Pro^1_W) \ar[rr]^{\partial} && H^1_{\text{\'{e}t}}(\Pro^1_W, H).       \\  }
\end{equation}
We have $\pr^*_W(s)=\varphi$. Hence
$$F= \partial (\varphi) = \partial (\pr^*_W(s)) = pr^*_W ( \partial (s)).$$
Setting $F_0=\partial (s)$ we see that
$F= pr^*_W ( F_0)$. The Lemma is proved.
\end{proof}

\begin{proof}[Proof of Theorem~\ref{ConstantOnFibreIsConstant}]
Let $U=Spec(\mathcal O)$. The $U$-group scheme $G$ is given by a $1$-cocycle
$\xi \in Z^1(U, Aut)$, where $Aut$ is the $k$-algebraic group from Section
\ref{EquatingGroups}
(the automorphism group of the split group $G_0$ from that Section).
Recall that
$Aut \cong G^{\ad}_0 \rtimes N$,
where $N$ is the finite group of automorphisms of the Dynkin diagram of $G_0$,
and $G^{\ad}_0$
is the adjoint group corresponding to $G_0$.
Since $Aut \cong G^{\ad}_0 \rtimes N$,
we have an exact sequence of pointed sets
$$\{1\} \to H^1(U, G^{\ad}_0) \to H^1(U, G^{\ad}_0 \rtimes N) \to H^1(U,N).$$
Thus there is a finite \'{e}tale morphism
$\pi: V \to U$ such that
$G_V:=G \times_U V$
is given by a
$1$-cocycle $\xi_V \in Z^1(U, G^{\ad}_0)$.

For each fundamental weight $\lambda$ of $G_0$, there is a central (also called center preserving, see~\cite{PS-tind})
representation
$\rho_{\lambda}: G_0 \to GL_{V_{\lambda},k}$,
where $V_{\lambda}$ is the Weyl module corresponding to $\lambda$.
This gives a commutative diagram of $k$-group morphisms
\begin{equation}
\label{Representation1}
    \xymatrix{
G_0 \ar[rr]^{\rho_{\lambda}} \ar[d]_{} && GL_{V_{\lambda},k} \ar[d]^{}& \\
G^{\ad}_0 \ar[rr]^{\bar \rho_{\lambda}} && PGL_{V_{\lambda},k}.      & \\  }
\end{equation}
Considering the product of $\rho_{\lambda}$'s with $\lambda$ running over the set $\Lambda$ of all fundamental
weights, we obtain the following commutative diagram
of algebraic $k$-group homomorphisms:
\begin{equation}
\label{Representation2}
    \xymatrix{
G_0 \ar[rr]^{\rho} \ar[d]_{} && \prod_{\lambda \in \Lambda}GL_{V_{\lambda},k} \ar[d]^{}& \\
G^{\ad}_0 \ar[rr]^{\bar \rho} && \prod_{\lambda \in \Lambda}PGL_{V_{\lambda},k}.      & \\  }
\end{equation}
Note that $\rho$ is a closed embedding. This fact will be used below.

Twisting the $k$-group morphism $\rho$ with the $1$-cocycle $\xi_V$ we get
an $V$-group scheme morphism
$\rho_V: G_V \to \prod_{\lambda \in \Lambda} GL_{1, A_{\lambda}}$,
where the product is taken over $V$, and each
$A_{\lambda}$
is an Azumaya algebra over $V$ obtained from
$End(V_{\lambda})$
via the $1$-cocycle
$\theta_{\lambda}=(\bar \rho_\lambda)_*(\xi_V) \in Z^1(V, PGL_{V_{\lambda}})$.
Set
$H=\prod_{\lambda \in \Lambda} GL_{1, A_{\lambda}}$.
One has
$$Hom_V(G_V,H)=Hom_U(G,R_{V/U}(H)),$$
where $R_{V/U}$ is the Weil restriction functor. Thus $\rho_V$ determines
an $U$-morphism
$$\rho_U: G \hra R_{V/U}(H).$$
Well-known properties of the functor $R_{V/U}$ imply that $\rho_U$ is a $U$-group scheme morphism.
Note that, since $\rho$ is a closed embedding, $\rho_U$ is a closed embedding as well
(by \'{e}tale descent).

Recall that we are given with a principal $G$-bundle $E \in H^1(\Pro^1, G)$.
The map
$$(\rho_U)_*: H^1(\Pro^1, G) \to H^1(\Pro^1, R_{V/U}(H))
$$
produces a principal
$R_{V/U}(H)$-bundle
$(\rho_U)_*(E)$
over $\Pro^1$.

By the Shapiro---Faddeev Lemma (for example,~\cite[Exp. XXIV Prop. 8.4]{SGA3}),
there is an isomorphism of pointed sets
\begin{equation}
\label{Faddeev}
    \xymatrix{
H^1(\Pro^1, R_{V/U}(H))\cong H^1(\Pro^1 \times_U V, H).
}
\end{equation}
Let $F$ be a principal $H$-bundle over $\Pro^1 \times_U V$ corresponding to the bundle
$(\rho_U)_*(E)$
via the latter isomorphism.
Recall that
$H= \prod_{\lambda \in \Lambda} GL_{1, A_{\lambda}}$.
Consider the compositions of the closed embeddings
$$j: H \hra \prod_{\lambda \in \Lambda} GL_{A_{\lambda}} \hra GL_{\oplus A_{\lambda}},$$
where
$GL_{A_{\lambda}}$ (resp. $GL_{\oplus A_{\lambda}}$)
is the $V$-group scheme of automorphisms of
$A_{\lambda}$ (resp. $\oplus A_{\lambda}$)
regarded as an $\mathcal O_V$-module.

Recall that the original $G$-bundle $E$ is trivial on
$\Pro^1_l= \Pro^1 \times_U x \subset \Pro^1$.
It follows that the $R_{V/U}(H)$-bundle $(\rho_U)_*(E)$ is trivial on
$\Pro^1 \times_U x \subset \Pro^1$.
Hence the
$H$-bundle $F$ is trivial on
$\Pro^1 \times_U \pi^{-1}(x)\subseteq \Pro^1 \times_U V$.
Hence the
$GL_{\oplus A_{\lambda}}$-bundle
$j_*(F)$ is trivial on
$\Pro^1 \times_U \pi^{-1}(x)$.
The $V$-group scheme
$GL_{\oplus A_{\lambda}}$
is just the ordinary general linear group, since $V$ is semi-local.
Thus $j_*(F)$ corresponds to a vector bundle $M$ over $\Pro^1 \times_U V$.
Moreover, this vector bundle is trivial on
$\Pro^1 \times_U \pi^{-1}(x)$.
By Horrocks' theorem the bundle $M$ is trivial on $\Pro^1 \times_U V$.
Thus
$j_*(F)$ is a trivial $GL_{\oplus A_{\lambda}}$-bundle.

Now, applying Lemma~\ref{Horrocks} to the embedding
$j:H \hra GL_{\oplus A_{\lambda}}$,
we see that $F=pr^*(F_0)$ for some
$F_0 \in H^1(V, H)$. Since
$H=\prod_{\lambda \in \Lambda} GL_{1, A_{\lambda}}$
and $V$ is semi-local, the bundle $F_0$ is trivial by Hilbert 90 for Azumaya algebras.
Whence $F$ is trivial as well.
The isomorphism (\ref{Faddeev})
shows that
$(\rho_U)_*(E)$
is trivial too.
Now, applying Lemma~\ref{Horrocks} to $W=U$, $j=\rho_U: G \hra R_{V/U}(H)$ and $F=E$,
we see that there exists a principal $G$-bundle $E_0$ over $U$ such that
$pr^*_U(E_0)\cong E$. The theorem is proved.

\end{proof}

\end{document}